\documentclass{article}

\usepackage{arxiv}

\usepackage[utf8]{inputenc} % allow utf-8 input
\usepackage[T1]{fontenc}    % use 8-bit T1 fonts
\usepackage{url}            % simple URL typesetting
\usepackage{booktabs}       % professional-quality tables
\usepackage{amsfonts}       % blackboard math symbols
\usepackage{nicefrac}       % compact symbols for 1/2, etc.
\usepackage{microtype}      % microtypography
\usepackage{lipsum}

\usepackage[unicode]{hyperref}

\usepackage{color}
\usepackage{latexsym}
\usepackage{indentfirst}
\usepackage{amsxtra}
\usepackage{amssymb}
\usepackage{amsthm}
\usepackage{amsmath}
\usepackage{mathrsfs} 
\usepackage[capitalise]{cleveref}

\newtheorem{theor}{Theorem}
\newtheorem*{theor*}{Theorem}
\newtheorem{prop}[theor]{Proposition}

\newtheorem*{cor*}{Corollary}
\theoremstyle{definition}               %stile roman 
\newtheorem{defin}[theor]{Definition}
\newtheorem{ex}[theor]{Example}
\newtheorem{exs}[theor]{Examples}
\newtheorem{rem}[theor]{Remark}

\DeclareMathOperator{\id}{id}
\DeclareMathOperator{\im}{Im}

\DeclareMathOperator{\M}{M}
\DeclareMathOperator{\lcm}{lcm}
\DeclareMathOperator{\ind}{i}
\DeclareMathOperator{\per}{p}

\newcommand{\s}[1]{S_{#1}}
\newcommand{\phii}[2]{\phi_{#1,#2}}

\newcommand{\X}[1]{X_{#1}}

\newcommand{\lambdaa}[2]{\lambda_{#1}{#2}}

\newcommand{\rhoo}[2]{\rho_{#1}{#2}}

\newcommand{\lambdai}[3]{\lambda^{#1}_{#2}{#3}}
\newcommand{\rhoi}[3]{\rho^{#1}_{#2}{#3}}
\newcommand{\indd}[1]{\ind{#1}}
\newcommand{\perr}[1]{\per{#1}}

\title{Set-theoretic solutions to the Yang-Baxter equation and generalized semi-braces}

\author{
  Francesco~Catino%\thanks{} 
  \\
  Dipartimento di Matematica e Fisica "Ennio De Giorgi"\\
  Università del Salento\\
    Via Provinciale Lecce-Arnesano\\
    73100 Lecce (Italy)\\
  \texttt{franceso.catino@unisalento.it} \\
   \And
 Ilaria ~Colazzo \\
 Department of Mathematics\\
    Vrije Universiteit Brussel, Pleinlaan 2\\
    1050 Brussel (Belgium)\\
  \texttt{ilaria.colazzo@vub.be} \\
  \And
 Paola ~Stefanelli \\
  Dipartimento di Matematica e Fisica "Ennio De Giorgi"\\
  Università del Salento\\
    Via Provinciale Lecce-Arnesano\\
    73100 Lecce (Italy)\\
  \texttt{paola.stefanelli@unisalento.it} \\
}
\usepackage{setspace}

\begin{document}
\maketitle

\begin{abstract}
This paper aims to introduce a construction technique of set-theoretic solutions of the Yang-Baxter equation, called \emph{strong semilattice of solutions}. This technique, inspired by the strong semilattice of semigroups, allows one to obtain new solutions. In particular, this method turns out to be useful to provide non-bijective solutions of finite order. It is well-known braces, skew braces and semi-braces are closely linked with solutions. Hence, we introduce a generalization of the algebraic structure of semi-braces based on this new construction technique of solutions.
\end{abstract}
% keywords can be removed
\keywords{Quantum Yang-Baxter equation  \and set-theoretical solution \and brace \and skew brace \and semi-brace \and generalized semi-brace\\
\textbf{MSC 2020} \quad 16T25 \and 81R50\and 16Y99 \and 16N20}

\doublespacing

\section{Introduction}
The quantum Yang-Baxter equation appeared in the work of Yang \cite{Ya67} and Baxter \cite{Ba72}. It is one of the basic equations in mathematical physics, and it laid the foundations of the theory of quantum groups \cite{Ka95}. Solutions of the Yang-Baxter equation are instrumental in the construction of semisimple Hopf algebras \cite{EtGe98, Ra12} and provide examples of colouring invariants in knot theory \cite{P06Y}.  More recently, the Yang-Baxter solution popped up in the theory of quantum computation \cite{kauffman2004braiding, zhang2005universal}, where solutions of the Yang-Baxter equation provide so-called universal gates.
One of the central open problems is to find all solutions of the Yang-Baxter equation. 
Let $V$ be a vector space over a field $K$. Then, a solution of the Yang-Baxter equation is a linear map $R: V \otimes V \longrightarrow V \otimes V$, for which the following holds on $V^{\otimes 3}$: $(R \otimes \id_V) (\id_V \otimes R) (R \otimes \id_V) = (\id_V \otimes R) (R \otimes \id_V) (\id_V \otimes R)$.
The simplest solutions are the solutions $R$ induced by a linear extension of a mapping $r: X \times X \longrightarrow X \times X$, where $X$ is a basis for $V$, satisfying the set-theoretic version of the Yang-Baxter equation, i.e., satisfying the following on $X^3$:
\begin{align*}
    (r \times \id_X) (\id_X \times r) (r \times \id_X) = (\id_X \times r) (r \times \id_X) (\id_X \times r).
\end{align*}
In this case, $r$ is said to be a \emph{set-theoretic solution of the Yang–Baxter equation} (briefly a \emph{solution}). Drinfel'd, in \cite{Dr90}, posed the question of finding these set-theoretic solutions.
Denote for $x,y \in X$, the element $r(x,y) = (\lambda_x(y),\rho_y(x))$. One says that a set-theoretic solution $r$ is \emph{left non-degenerate} if $\lambda_x$ is bijective, for every $x\in S$, \emph{right non-degenerate} if $\rho_y$ is bijective, for every $y\in S$, and \emph{non-degenerate} if $r$ is both left and right non-degenerate.
If a solution is neither left nor right non-degenerate, then it is called \emph{degenerate}.
The first papers on set-theoretic solutions are those of Etingof, Schedler and Soloviev \cite{ESS99} and Gateva–Ivanova and Van den Bergh \cite{GaB98}. Both papers considered \emph{involutive} solutions, i.e.,  solutions $r$ where $r^2= \id$. Rump \cite{Ru07a} introduced a new algebraic structure, braces, that generalizes radical rings and provides an algebraic framework. We provide the equivalent definition formulated by Ced\'o, Jespers and Okni\'nski \cite{CeJO14}. A triple $(B,+,\circ)$ is called a \emph{left brace} if $(B,+)$ is an abelian group and $(B,\circ)$ is a group, such that for any $a,b,c \in B$ it holds that 
\begin{align}\label{eq:brace}\tag{$\diamond$}
a \circ (b+c) = a\circ b -a + a \circ c.
\end{align}
This new structure showed connections between the Yang-Baxter equation and ring theory, flat manifolds, orderability of groups, Garside theory, regular subgroups of the affine group, see for instance \cite{CCoSt15, CCoSt16, Ch16, ChGo14, De15, Ga12, Sm18}.
Lu, Yan, and Zhu \cite{LuYZ00} and Soloviev \cite{So00} started the study of non-degenerate bijective solutions, not necessarily involutive.
Almost all of the ideas used in the theory of involutive solutions can be transported to non-involutive solutions. The algebraic framework now is provided by skew left braces \cite{GVe17}. Let $(B,+)$ and $(B,\circ)$ be groups on the same set $B$. If, for any $a,b,c \in B$, condition \eqref{eq:brace} holds, the triple $(B,+,\circ)$ is called a \emph{skew left brace}.
Skew left braces and some of their applications are intensively studied,  see for instance \cite{CCoSt19, Ch18, DeC19, KoTr20, Ru19}.

In \cite{Le17}, Lebed drew the attention on idempotent solutions. Indeed, using idempotent solutions and graphical calculus from knot theory, she provides a unifying tool to deal with several diverse algebraic structures, such as free monoids, free commutative monoids, factorizable monoids, plactic monoids and distributive lattice. Examples and classifications of these solutions have been provided by Matsumoto and Shimizu \cite{MaShi18} and by Stanovsk{\`y} and Vojt{\v{e}}chovsk{\`y} \cite{StVo20x}. Moreover, Cvetko-Vah and Verwimp, \cite{CvVer20} provided cubic solutions with skew lattices. A cubic solution $r$ is a solution such that $r^3=r$; hence, this class includes both involutive and idempotent solutions.
More in general, a more systematic approach to the study of solutions with finite order
can be found recently in \cite{CCoSt20, CCoSt20-2,CaMaSt20x}.
Catino, Colazzo, and Stefanelli \cite{CCoSt17} and Jespers and Van Antwerpen \cite{JVA19} introduced the algebraic structure called \emph{left semi-brace} to deal with solutions that are not necessarily non-degenerate or that are idempotent.
Let $(B,+)$ be a semigroup and $(B,\circ)$ be a group. Then $(B,+,\circ)$ is called a left semi-brace if, for any $a,b,c \in B$, it holds that 
\begin{align*}
    a \circ (b+c) = a \circ b + a \circ (a^- +c),
\end{align*}
where $a^-$ denotes the inverse $a$ in $(B,\circ)$. If $(B,+)$ is a left cancellative semigroup, then we call $(B,+,\circ)$ a left cancellative left semi-brace. This was the original definition by Catino, Colazzo and Stefanelli \cite{CCoSt17}. It has been shown that left semi-braces, under some mild assumption, provide set-theoretic solutions of the Yang-Baxter equation. Moreover, the associated solution is left non-degenerate if and only if the left semi-brace is left cancellative. 

Out of algebraic interest Brzezi{\'n}ski introduced left trusses \cite{Br19} and left semi-trusses \cite{Br18}. A quadruple $(B,+,\circ,\lambda)$ is called a \emph{left semi-truss} if both $(B,+)$ and $(B,\circ)$ are semigroups and $\lambda:B \times B \rightarrow B$ is a map such that $a \circ (b+c) = (a \circ b) + \lambda(a,c)$.
Clearly, the class of left semi-trusses contains all left semi-braces, rings, associative algebras and distributive lattices. This  entails  that  it  will  prove  difficult  to  present  deep  results  on  this  class.   However,  one  may examine  large subclasses.   In  particular,  Brzezi{\'n}ski  \cite{Br19} focused  on left semi-trusses with  $(B,+)$ a left cancellative semigroup and $(B,\circ)$ a group, and showed that such a left semi-truss is equivalent with a left cancellative semi-brace, thus providing set-theoretic solutions of the Yang-Baxter, albeit known ones.
In \cite{Mi18} Miccoli introduced almost left semi-braces, a particular instance of left semi-trusses, and constructed set-theoretic solutions associated with this algebraic structure. 
In \cite{CVA20x}, Colazzo and Van Antwerpen continue this study focusing on the subclass of brace-like left semi-trusses, i.e., left semi-trusses in which the multiplicative semigroup is a group and which includes almost left semi-braces. Concerning solutions, they showed that the solution one can associate with an almost left semi-brace is already the associated solution of a left semi-brace. In particular, this shows that brace-like left semi-trusses will not yield a universal algebraic structure that produces set-theoretic solutions. 

In this paper, we focus on a new algebraic structure that includes left semi-braces and is an instance of left semi-trusses, which is on a different path with respect to brace-like left semi-trusses, called generalized left semi-brace. A triple $(S,+,\circ)$ is called a \emph{generalized left semi-brace} if $\left(S,+\right)$ is a semigroup, $\left(S,\circ\right)$ is a completely regular semigroup (or union of groups),
and such that, for any $a,b,c\in S$. it holds
\begin{align*}
a \circ \left(b + c\right) = a \circ b + a\circ\left(a^- + c\right),
\end{align*}
where $a^-$ denotes the (group) inverse of $a$ in $\left(B,\circ\right)$. 
We prove that under some mild assumptions, generalized left semi-braces provide solutions. In particular, elementary examples of generalized left semi-braces produce cubic solutions that cannot be obtained by skew lattices and left semi-brace. Also, we introduce a construction technique that provides generalized left semi-braces called the strong semilattice of generalized left semi-braces. This technique is inspired by the description of semigroups which are unions of groups due to Clifford \cite{Cl41}. 

Furthermore, we introduce a construction technique for solutions called the strong semilattice of solutions. This technique takes a family of disjoint sets $\left\{X_{\alpha}\ |\ \alpha \in Y\right\}$ indexed by a semilattice $Y$ and solutions defined on these sets, then, under some assumptions of compatibility, it allows one to construct a solution on the union of the sets $X_{\alpha}$. We prove that the solutions provide by the strong semilattice of left semi-braces are a particular instance of strong semilattice of solutions. 

Finally, we prove that the strong semilattice of solutions is a useful tool to provide solutions of finite order. Indeed the strong semilattice of solutions of finite order is a solution of finite order. Moreover, a solution $r$ is of finite order if there exist a non-negative integer $i$ and a positive integer $p$ such that $r^{p+i}=r^i$ and the minimal integers that satisfy such relation are said to be index and period respectively. We show that it is possible to determine the index and the period of the semilattice of solutions $\left\{r_{\alpha}\ |\ \alpha \in Y\ \right\}$ as a function of the indexes and periods of $r_{\alpha}$. As a corollary of this result, we prove that solutions associated with strong semilattices of left semi-braces are not bijective, so they are clearly different from solutions obtained by left semi-braces.

% ------------------------------------------
\section{Basic tools on left semi-braces}
% ------------------------------------------

Let us briefly present some basic background information regarding left semi-braces. Most of the content of this section appear in \cite{JVA19}. In particular, we provide different proof of \cite[Corollary2.9]{JVA19} based on a result in semigroup theory due to Hickey \cite{Hi83} that gives a clear description of completely regular semigroup with middle units. Moreover, we add further information on the behavior of middle units of the additive semigroup of a left semi-brace. Finally, we present concrete examples of left semi-braces.

Let's start by recalling the definition of left semi-braces. 

\begin{defin}
	Let $B$ be a set with two operations $+$ and $\circ$ such that $\left(B,+\right)$ is a semigroup and $\left(B,\circ\right)$ is a group. Then, $\left(B, +, \circ \right)$ is said to be a \emph{left semi-brace} if
	\begin{align*}
		a \circ \left(b+c\right) = a \circ b + a\circ\left(a^- +c\right),
	\end{align*}
for all $a, b, c \in B$, where $a^-$ is the inverse of $a$ in $\left(B, \circ \right)$. 
\end{defin}
\noindent Throughout, $0$ denotes the identity of the group  $\left(B, \circ \right)$. Moreover, we call $\left(B, +\right)$  and $\left(B, \circ\right)$ the \emph{additive semigroup} and the \emph{multiplicative group} of the left semi-brace $\left(B, + , \circ \right)$, respectively. 
Furthermore, if the semigroup $\left(B,+\right)$ has a pre-fix, pertaining to some property of the semigroup, we will also use this pre-fix with the left semi-brace. Hence, the left semi-braces introduced in \cite{CCoSt17}, where one works under the restriction that the semigroup $\left(B,+\right)$ is left cancellative, will be called left cancellative left semi-braces.

Now, we recall that an element $u$ of an arbitrary semigroup $\left(S,+\right)$ is a \emph{middle unit} of $S$ if 
$a + u +b = a+ b$,
for all $a, b\in S$. Thus, $u+u$ is idempotent but $u$ itself need not be idempotent (see \cite[p. 98]{ClPr61}).
This is not the case for the element $0$  in the additive semigroup of a left semi-braces: the following proposition shows that $0$ is an idempotent middle unit. 
\begin{prop}\label{prop:prop-semi-brace}
	Let $B$ be a left semi-brace. Then, the following hold:
	\begin{enumerate}
		\item $0$ is a middle unit of $\left(B,+\right)$;
		\item $0$ is an idempotent of $\left(B,+\right)$;
		\item $B = B + B$;
		\item $B + 0$ 
		is a subgroup of $\left(B,\circ\right)$;
		\item $0+B$ 
		is a subsemigroup of $\left(B,\circ\right)$.
	\end{enumerate}
	\begin{proof}
		\begin{enumerate}
			\item See \cite[Lemma 2.4(1)]{JVA19}.
			\item Since, by $(1)$, $0 + 0 + 0 = 0 + 0$, we have that 
			\begin{align*}
				0 
				&= \left(0 + 0\right)^-\circ \left(0 + 0\right) 
				= \left(0 + 0\right)^-\circ 0 + \left(0 + 0\right)^-\circ\left(0 + 0 + 0\right)\\
				&= \left(0 + 0\right)^- + \left(0 + 0\right)^-\circ\left(0 + 0\right)  
				= \left(0 + 0\right)^- + 0.
			\end{align*}
			Thus, since $0$ is a middle unit,  
			\begin{align*}
				0 + 0 = \left(0 + 0\right)^- + 0 + 0 = \left(0 + 0\right)^- + 0 = 0.
			\end{align*}
			\item If $b\in B$, by $(2)$ we have that
			\begin{align*}
				b 
				= b\circ 0 
				= b\circ\left(0+0\right)
				= b\circ 0 +  b\circ\left(b^- + 0\right)\in B + B.
			\end{align*} 
			\item 
			By $(2)$ it is clear that $B+0$ is not empty. Moreover, by $(1)$, if $a,b\in B$, it holds that
			\begin{align*}
				\left(a + 0\right)^-&\circ \left(b + 0\right)
				= \left(a + 0\right)^-\circ b + \left(a + 0\right)^-\circ\left(a + 0 + 0\right)\\
				&= \left(a + 0\right)^-\circ b + \left(a + 0\right)^-\circ\left(a + 0\right)
				= \left(a + 0\right)^-\circ b + 0\in B+0.
			\end{align*}
			\item See \cite[Lemma 2.6(iii)]{JVA19}. 
		\end{enumerate}
	\end{proof}
\end{prop} 

To show the following theorem, let us recall that an arbitrary semigroup $\left(S,+\right)$ is said to be a \emph{rectangular group} if it is isomorphic to the direct product of a group and a rectangular band. For background and details on this topic we refer the reader to \cite{ClPr61}.
\begin{theor}\label{th:rect-group}
	Let $B$ be a completely simple left semi-brace. Then, the additive semigroup $\left(B,+\right)$ of $B$ is a rectangular group.
	\begin{proof}
		The thesis follows by \cite[Corollary 3.5]{Hi83}, which states that any completely simple semigroup with a middle unit is a rectangular group.
	\end{proof}
\end{theor}
\noindent A special case in which the additive semigroup is completely simple is when $0+B$ is a subgroup of $\left(B, \circ\right)$ (see \cite[Theorem 2.8]{JVA19}). For instance, this is the case when $B$ is finite.\\
As a consequence of \cref{th:rect-group}, we have that the additive semigroup $\left(B,+\right)$ can be written as
\begin{align*}
	B = I + G+ \Lambda,
\end{align*}
that is the direct sum of the left zero semigroup $I:= E\left(B + 0\right)$,  the group 
$G:= 0 + B + 0$, and the right zero semigroup $\Lambda:= E\left(0+ B\right)$.
Moreover, the set of idempotents is $E\left(B\right) = I + \Lambda$ and it is a rectangular band. Where $E\left(S\right)$ denotes the set of idempotents of $S$, for any semigroup $S$.
\medskip

For the sake of completeness, let us introduce a further property of middle units of left semi-braces that holds without restrictions on the additive semigroup.
At first, we recall that the additive semigroup of a left semi-brace $B$ does not contain a zero element if $B$ has at least two elements (see \cite[Lemma 2.3]{JVA19}).\\
In the following, we prove that middle units of the additive structure of an arbitrary left semi-brace are idempotents. 
\begin{prop}\label{prop:middle-ide}
	Let $B$ be a left semi-brace and $e\in B$. Then, if $e$ is a middle unit then $e$ is an idempotent of the semigroup $\left(B,+\right)$.
	\begin{proof}
		Since $0$ is idempotent, we have that
		\begin{align*}
			e = e\circ 0
			= e\circ\left(0 + 0\right)
			= e\circ 0 + e\circ\left(e^- + 0\right)
			= e + e\circ\left(e^- + 0\right).
		\end{align*}
		Since $e$ is a middle unit, it follows that
		\begin{align*}
			e + e = e + e + e\circ\left(e^- + 0\right)
			= e + e\circ\left(e^- + 0\right)
			= e,
		\end{align*}
		which is our assertion.
	\end{proof}
\end{prop}
\medskip 

Following Ault's paper \cite[Theorem 1.8]{Au74}, by \cref{prop:middle-ide} we have that the additive semigroup of any left semi-brace contains a subsemigroup $\M_B$, called the \emph{semigroup of middle units of $B$}, that is explicitly given by 
\begin{align*}
\M_B= \lbrace x \ | \ x\in B, \ x \ \text{has inverse} \ x' \ \text{with} \ x + x', x' + x \ \text{middle units}\rbrace ,
\end{align*}
where the inverse $x'$ of $x$ means that $x' = x' + x + x'$ and $x = x + x' + x$.\\
Then, $M_B$ is a subsemigroup of $\left(B,+\right)$  that is a rectangular group. 
This is an interesting substructure of a left semi-brace, which is beyond the purpose of this paper and shall be studied elsewhere.
\medskip

Now, having as reference the \cite[Example 2]{CCoSt17}, we provide the following class of examples of completely simple left semi-braces that allow one to obtain solutions.
\begin{ex}
	Let $\left(B,\circ\right)$ be a group with identity $0$, $f,g$ idempotent endomorphisms of $\left(B,\circ\right)$ such that $fg=gf$. Let us consider the following operation 
	\begin{align*}
		a +b := b\circ fg\left(b^-\right)\circ f\left(a\right),
	\end{align*}
	for all $a,b\in B$. It is easy to check that $\left(B, +,\circ\right)$ is a completely simple left semi-brace. 
	Now, observe that the map $\rho$ is an anti-homomorphism from the group $\left(B,\circ\right)$ into the monoid $B^B$, where $B^B$ denotes the monoid of the functions from $B$ into itself. Indeed,
	$\rho_b$ is given by
	\begin{align*}
		\rho_b\left(a\right) 
		= \left(a^- + b\right)^-\circ b
		= \left(b\circ fg\left(b^-\right)\circ f\left(a^-\right)\right)^-\circ b
		= f\left(a\right)\circ fg\left(b\right),
	\end{align*} 
	moreover
	\begin{align*}
		\rho_{b}\rho_{a}\left(c\right)
		= \rho_{b}\left(f\left(c\right)\circ fg\left(a\right)\right)
		=  f^2\left(c\right)\circ f^2g\left(a\right)\circ fg\left(b\right)
		= f\left(c\right)\circ fg\left(a\circ b\right)
		= \rho_{a\circ b}\left(c\right),
	\end{align*}
	for all $a,b,c\in B$.
	Thus, by \cite[Proposition 2.14]{JVA19}, the semigroup $\left(B, +\right)$ is completely simple. 	\\
	In addition, since $\rho$ is an anti-homomorphism of the group $\left(B,\circ\right)$,  we obtain that the map $r:B\times B\to B\times B$ defined by $r\left(a,b\right) = \left(\lambda_a\left(b\right), \rho_b\left(a\right)\right)$ in \cite[Theorem 5.1]{JVA19} is a solution.
	Note that $\lambda_a\left(b\right) = a\circ b\circ fg\left(b^-\right)\circ f\left(a^-\right)$,  
	therefore, $r$ is explicitly given by
	\begin{align*}
		r\left(a, b\right)
		= \left(a\circ b\circ f\left(g\left(b^-\right)\circ a^-\right),
		f\left(a\circ g\left(b\right)\right)\right),
	\end{align*}
	for all $a,b\in B$.
\end{ex}

Let us examine special cases of the previous class of examples. 
Firstly, observe that if $f\neq \id$ and $g$ is not the constant map of value $0$, then the semigroup $\left(B,+\right)$ is not neither left nor right cancellative. 
Moreover, note that:\\
\noindent \textbf{Case 1}: if $g$ is the constant map of value $0$, then 
\begin{align*}
	a + b = b\circ f\left(0\right)\circ f\left(a\right) = b\circ f\left(a\right),
\end{align*}
for all $a, b\in B$, i.e.,  $B$ coincides with the left cancellative  left semi-brace provided in \cite[Example 2]{CCoSt17}.
Moreover,  the solution associated to $B$ is given by
\begin{align*}
	r\left(a, b\right)
	= \left(a\circ b\circ f\left(a^-\right),
	f\left(a\right)\right).
\end{align*}
\textbf{Case 2}: if $f =\id$, then
\begin{align*}
	a + b = b\circ g\left(b^-\right)\circ a,
\end{align*}
for all $a, b\in B$. In this case, the semigroup $\left(B,+\right)$ is right cancellative. Indeed, if 
$a + b = c + b$, it follows that  $b\circ g\left(b^-\right)\circ a 
= b\circ g\left(b^-\right)\circ c$, and so $a = c$.
Moreover, it is easy to check that $B$ is both right and left semi-brace.
In addition, note that the solution associated to $B$ is given by
\begin{align*}
	r\left(a, b\right)
	= \left(a\circ b\circ g\left(b^-\right)\circ a^-,
	a\circ g\left(b\right)\right).
\end{align*}
\textbf{Case 3}: if $f = g$, then
\begin{align*}
	a + b = b\circ f^2\left(b^-\right)\circ f\left(a\right) = b\circ f\left(b^-\right)\circ f\left(a\right),
\end{align*}
for all $a, b\in B$. Note that  if $a\in B$, it holds 
\begin{align*}
	a + a =  a\circ f\left(a^-\right)\circ f\left(a\right) = a,
\end{align*}
i.e., every element is idempotent with respect to the sum. 
In this case, the semigroup $\left(B, +\right)$ is a rectangular band where $\ker f = 0 + B$ and $\im f= B + 0$.
Moreover, the solution $r$ associated to $B$, given by
\begin{align*}
	r\left(a,b\right) 
	= \left(a\circ b\circ f\left(b^-\circ a^-\right), 
	f\left(a\circ b\right)
	\right),
\end{align*}
is an idempotent solution, consistently with \cite[Theorem 5.1]{JVA19} in the case in which  the group $G = \lbrace0\rbrace$. 
\bigskip

The following is a class of examples of completely simple left semi-braces such that, under suitable assumptions, give rise to solutions.
\begin{ex}
	Let $G$, $H$ be two groups, $B:=G\times H$ and consider the group $\left(B,\circ\right)$ where
	\begin{align*}
		\left(a,u\right)\circ \left(b,v\right) = \left(a \ ^{u}b, \ u^b \  v\right),
	\end{align*}
	for all $\left(a,u\right), \left(b,v\right)\in G\times H$, i.e., the classical Zappa-Sz\'{e}p product of $G$ and $H$ (see \cite{Ku83}) that has identity $\left(1,1\right)$.
	Let $\varphi$ be a map from $G$ into $H$ such that 
	$\varphi\left(1\right) = 1$ and define the following operation on $B$
	\begin{align*}
		\left(a,u\right) + \left(b,v\right) = \left(a, \ u\varphi\left(b\right)v\right),
	\end{align*}
	for all $\left(a,u\right),\left(b,v\right)\in G\times H$. 
	It is easy to check that the structure $\left(B,+,\circ\right)$ is a left semi-brace.
	Let us note that $\left(B,+\right)$ is a left group.
	Moreover, $\left(a,u\right)$ in $G\times H$ is idempotent with respect to the sum if and only if $\varphi\left(a\right) = u^{-1}$.
	In addition, if $\left(a,u\right), \left(b,v\right)\in E\left(B\right)$, we have that
	\begin{align*}
		\left(a,u\right) + \left(b,v\right)= \left(a, u\varphi\left(b\right)v\right)
		= \left(a, u\right),
	\end{align*}
	hence $E\left(B\right)$ is a sub-semigroup of $\left(B,+\right)$ and it is also a left zero semigroup. Note also that 
	$\left(a,u\right) + \left(1,1\right) 
	= \left(a, u\varphi\left(1\right)1\right)
	= \left(a, u\right)$,
	i.e., $\left(1,1\right)$ is a right identity with respect to the sum. 
	Furthermore, we have that
	\begin{align*}
		%&\lambda_{\left(a,u\right)}\left(b,v\right)
		%= \left(0, \varphi\left(b\right)v\right)\\
		%&
		\rho_{\left(b,v\right)}\left(a,u\right)
		= \left({}^{\left(\varphi\left(b\right)v\right)^{-1}}\left(a \ ^{u}{b}\right), \ 
		\left(\left(\left(\varphi\left(b\right)v\right)^{-1}\right)^{a} \ u\right)^{b}\ v\right),
	\end{align*}
	for all $\left(a,u\right), \left(b,v\right), \left(c,w\right)\in B$. 
	One can check that $\rho$ is an anti-homomorphism if and only if it holds 
		\begin{align}\label{eq:rho-anti}
		\varphi\left(b\right)v\varphi\left(c\right)
		= \varphi\left(b \ ^{v}c\right)v^c,
		\end{align}
	for all $a,b\in G$ and $u\in H$. 
	Moreover, by the characterization \cite[Theorem 3]{CCoSt20-2}, one can verify that the map $r$ in \cite[Theorem 5.1]{JVA19} associated to the left semi-brace $B$ is a solution if and only if
	\begin{align}\label{eq:r-sol}
		\varphi\left(a\right)u\varphi\left(b\right)
		=\varphi\left(b\right) v \ \varphi\left(^{\left(\varphi\left(b\right)v\right)^{-1}}\left(a \ ^{u}b\right)\right)
		\left(\left(\left(\varphi\left(b\right)v\right)^{-1}\right)^{a} \ u \right)^b
	\end{align}
	holds, for all $a,b\in G$ and $u,v\in H$.
	On the other hand, note that, if $a,b\in G$ and $u\in H$, considered $v = \varphi\left(b\right)^{-1}$ in \eqref{eq:r-sol}, we obtain 
	\begin{align*}
		\varphi\left(a\right)u\varphi\left(b\right)
		= v^{-1} v \ \varphi\left(^{\left(v^{-1}v\right)^{-1}}\left(a \ ^{u}b\right)\right)
		\left(\left(\left(v^{-1}v\right)^{-1}\right)^{a} \ u \right)^b
		= \varphi\left(a \ ^{u}b\right)u^b,
	\end{align*}
	hence \eqref{eq:rho-anti} is satisfied.\\
	Now, let us consider $G$ the cyclic group $C_2$ of $2$ elements, $H$ the cyclic group $C_3$ of $3$ elements, and $\varphi$ the constant map of value $1$ from $G$ into $H$. 
	Hence, if $\left(B,\circ\right)$ is the cyclic group $C_2\times C_3$, then the condition \eqref{eq:rho-anti} trivially holds, hence $r$ is a solution. 
	Instead, if $\left(B,\circ\right)$ is the symmetric group $C_2\ltimes C_3$, then \eqref{eq:rho-anti} is not satisfied, equivalently \eqref{eq:r-sol} does not hold, hence $r$ is not a solution. 
\end{ex}

% -------------------------------------------
\section{Definitions and examples}
% -------------------------------------------

Braces, skew braces, and semi-braces were introduced to study set-theoretic solutions of the Yang-Baxter equation. The following definition generalizes these structure to the case in which the multiplicative structure is no more a group. 

	\medskip 
	
	At first, we recall that a semigroup $(S,\circ)$ is \emph{completely regular} if for any element $a$ of $S$ there exists a (unique) element $a^-$ of $S$ such that 
	\begin{equation}\label{cr}
	a= a\circ a^-\circ a , \quad a^-=a^-\circ a\circ a^-, \quad a\circ a^-=a^-\circ\, a.
	\end{equation}
	Conditions \eqref{cr} imply that $a^0:= a\circ a^-=a^-\circ a$ is an idempotent element of $(S,\circ)$.
	\begin{defin}
        Let $S$ be a set with two operations $+$ and $\circ$ such that $\left(S,+\right)$ is a semigroup (not necessarily commutative) and $\left(S,\circ\right)$ is a completely regular semigroup. Then, we say that $\left(S, + , \circ \right)$ is a \emph{generalized left semi-brace} if
        \begin{align}\label{key1}
            a \circ \left(b+c\right) = a \circ b + a\circ\left(a^- +c\right),
        \end{align}
        for all $a, b, c \in S$. 
        We call $(S,+)$ and $(S,\circ)$ the \emph{additive semigroup} and the 	\emph{multiplicative semigroup} of $S$, respectively.\\
        A \emph{generalized right semi-brace} is defined similarly, replacing condition \eqref{key1} by
        \begin{align}\label{key2}
            \left(a+b\right)\circ c =  \left(a+c^-\right )\circ c + b\circ c,
        \end{align}
        for all $a,b,c\in S$. \\
        A \emph{generalized two-sided semi-brace} is a generalized left semi-brace that is also a generalized right semi-brace with respect to the same pair of operations.
    \end{defin} 
\medskip

 Let us note that if $S$ is a generalized left semi-brace and $a\in S$, then the map
\begin{equation*}
		\lambda_a : S\longrightarrow S,\; b\longmapsto a\circ (a^-+b)
\end{equation*}
is an endomorphism of the semigroup $(S,+)$ and $\lambda_{a\circ b}(x)= (a\circ b)^0+ \lambda_a\lambda_b(x)$, for all $a,b,x\in S$.
Indeed, if $a,b,x,y\in S$, we have that
	\begin{align*}
		\lambda_a(x+y)  = a\circ (a^- +x +y) = a\circ (a^- + x) + a\circ (a^- +y)
		= \lambda_a(x) +\lambda_a(y).
	\end{align*}
	and 
	\begin{align*}
		\lambda_{a\circ b}(x) &= (a\circ b)\circ ((a\circ b)^-+ x)
		= a\circ (b\circ (a\circ b)^- + b\circ (b^-+x))\\
		&= a\circ b\circ (a\circ b)^- + a\circ (a^- +\lambda_b(x))\\
		&= (a\circ b)^0 + \lambda_a\lambda_b(x).
	\end{align*}

Of course, left semi-braces \cite{CCoSt17, JVA19} are examples of generalized left semi-braces. Moreover, a generalized left semi-brace can be obtained from every completly regular semigroup.
\begin{ex}
If $(S,\circ)$ is an arbitrary completely regular semigroup and $(S,+)$ is a right zero semigroup (or a left zero semigroup), then $(S,+,\circ)$ is a generalized two-sided semi-brace. 
\end{ex}

Unlike left semi-braces, a generalized left semi-brace $S$ can have a zero element even if $S$ has more then one element. Examples of such generalized left semi-braces can be easily obtained by any Clifford semigroup.
\begin{ex}\label{Cliff}
If $(S,\circ)$ is a Clifford semigroup, which is a completely regular semigroup where all idempotent elements are central, then  $(S,+,\circ)$, where $a+b=a\circ b$ for all $a,b\in S$, is a generalized two-sided semi-brace. 
\end{ex}

More in general, the previous generalized left semi-braces can be obtained through the following construction.
\begin{prop}\label{prop:StrongLatticeGeneralizedSemiBrace}
Let $Y$ be a (lower) semilattice, $\left\{S_{\alpha}\ \left|\ \alpha \in Y\right.\right\}$ a family of disjoint generalized left semi-braces. For each pair $\alpha,\beta$ of elements of $Y$ such that $\alpha \geq \beta$, let $\phii{\alpha}{\beta}:\s{\alpha}\to \s{\beta}$ be a homomorphism of generalized left semi-braces such that
\begin{enumerate}
    \item $\phii{\alpha}{\alpha}$ is the identical automorphism of $\s{\alpha}$, for every $\alpha \in Y$
    \item $\phii{\beta}{\gamma}{}\phii{\alpha}{\beta}{} = \phii{\alpha}{\gamma}{}$, for all $\alpha, \beta, \gamma \in Y$ such that $\alpha \geq \beta \geq \gamma$.
\end{enumerate}
Then, $S = \bigcup\left\{\s{\alpha}\ \left|\ \alpha\in Y\right.\right\}$ endowed by the addition and the multiplication defined by
\begin{align}\label{eq:newoperationGeneralizedSemiBrace}
    a+b= \phii{\alpha}{\alpha\beta}(a)+\phii{\beta}{\alpha\beta}(b),\\
     a\circ b= \phii{\alpha}{\alpha\beta}(a)\circ\phii{\beta}{\alpha\beta}(b),
\end{align}
for any $a\in \s{\alpha}$ and $b\in \s{\beta}$, is a generalized left semi-brace. Such a generalized left semi-brace is said to be the \emph{strong semilattice $Y$ of generalized left semi-brace $S_{\alpha}$} and is denoted by $S=[Y; S_\alpha,\phii{\alpha}{\beta}]$.
\end{prop}
    \begin{proof}
        First note that $\left(S,+\right)$ is a semigroup and $\left(S,\circ\right)$ is a completely regular semigroup. Now, let $a\in\s{\alpha}$, $b \in \s{\beta}$, and $c \in \s{\gamma}$. Set $\delta:=\alpha\beta$, $\varepsilon:=\beta\gamma$, $\zeta:=\alpha\gamma$, and $\eta:=\alpha\beta\gamma$. It follows that
        \begin{align*}
            a\circ\left(b+c\right)
            &=a\circ\left(\phii{\beta}{\varepsilon}\left(b\right)+\phii{\gamma}{\varepsilon}\left(c\right)\right)\\
            &=\phii{\alpha}{\eta}\left(a\right)\circ\phii{\varepsilon}{\eta}\left(\phii{\beta}{\varepsilon}\left(b\right)+\phii{\gamma}{\varepsilon}\left(c\right)\right)&\text{since $\alpha\varepsilon=\eta$}\\
            &=\phii{\alpha}{\eta}\left(a\right)\circ\left(\phii{\varepsilon}{\eta}\phii{\beta}{\varepsilon}\left(b\right)+\phii{\varepsilon}{\eta}\phii{\gamma}{\varepsilon}\left(c\right)\right)\\
            &=\phii{\alpha}{\eta}\left(a\right)\circ\left(\phii{\beta}{\eta}\left(b\right)+\phii{\gamma}{\eta}\left(c\right)\right)&\text{by 2.}\\
            &=\phii{\alpha}{\eta}\left(a\right)\circ\phii{\beta}{\eta}\left(b\right)+\phii{\alpha}{\eta}\left(a\right)\circ\left(\left(\phii{\alpha}{\eta}\left(a\right)\right)^-+\phii{\gamma}{\eta}\left(c\right)\right)
        \end{align*}
        where the last equality holds since $\s{\eta}$ is a generalized left semi-brace. Moreover,
        \begin{align*}
            a\circ b &+ a \circ\left(a^-+c\right)
            =\phii{\alpha}{\delta}\left(a\right)\circ\phii{\beta}{\delta}\left(b\right)+a\circ\left(\left(\phii{\alpha}{\zeta}\left(a\right)\right)^-+\phii{\gamma}{\zeta}\left(c\right)\right)\\
            &=\phii{\alpha}{\delta}\left(a\right)\circ\phii{\beta}{\delta}\left(b\right)+\phii{\alpha}{\zeta}\left(a\right)\circ\left(\left(\phii{\alpha}{\zeta}\left(a\right)\right)^-+\phii{\gamma}{\zeta}\left(c\right)\right)\\
            &=\phii{\delta}{\eta}\left(\phii{\alpha}{\delta}\left(a\right)\circ\phii{\beta}{\delta}\left(b\right)\right)+\phii{\zeta}{\eta}\left(\phii{\alpha}{\zeta}\left(a\right)\circ\left(\left(\phii{\alpha}{\zeta}\left(a\right)\right)^-+\phii{\gamma}{\zeta}\left(c\right)\right)\right)&\text{since $\delta\zeta=\eta$}\\
            &=\phii{\alpha}{\eta}\left(a\right)\circ\phii{\beta}{\eta}\left(b\right)+\phii{\alpha}{\eta}\left(a\right)\circ\left(\left(\phii{\alpha}{\eta}\left(a\right)\right)^-+\phii{\gamma}{\eta}\left(c\right)\right)&\text{by 2.}
        \end{align*}
        Therefore, $S$ is a generalized left semi-brace.
    \end{proof}
    
\begin{rem}
    If $S=[Y; S_{\alpha},\phii{\alpha}{\beta}]$ is a strong semilattice $Y$ of left semi-braces $S_{\alpha}$, then $\left(S,\circ\right)$ is a strong semilattice of groups, and hence, by \cite[Theorem 4.2.1]{Ho76book}, $\left(S,\circ\right)$ is a Clifford semigroup.
\end{rem}

% ---------------------------------------------------
\section{Solutions related to generalized left semi-braces}
% ---------------------------------------------------

This section is devoted to provide a sufficient condition to obtain solutions through a generalized left semi-brace. 
To this end, we recall that if $S$ is a left cancellatice left semi-brace, then the map $r:S\times S\to S\times S$ given by
\begin{align}\label{eq:r-solution}
r\left(a,b\right):=\left(a\circ\left(a^-+b\right), \left(a^-+b\right)^-\circ b\right)  
\end{align}
for all $a, b\in S$, is a solution. 
Moreover, \cite[Theorem $5.1$]{JVA19} gives a sufficient condition to obtain that 
the map in \eqref{eq:r-solution} is still a solution for a left semi-brace, not necessarily left cancellative.
In addition, in \cite[Theorem 3]{CCoSt20-2}, we state a necessary and sufficient condition to ensure that $r$ is a solution.

Specifically, if $\left(S,+,\circ\right)$ is a left semi-brace, then, the  map $r:S\times S \to S\times S$, defined by $r\left(a,b\right):=\left(a\circ\left(a^-+b\right), \left(a^-+b\right)^-\circ b\right)$, for all $a,b \in S$, is a solution if and only if 
	\begin{align}\label{eq:condsolutionsemi}
	a + \lambdaa{b}{\left(c\right)}\circ\left(0 + \rhoo{c}{\left(b\right)}\right) = a + b\circ \left(0+c\right)
	\end{align}
holds, for all $a,b,c\in S$.

Let us remark that if $S$ is a generalized left semi-brace with $\left(S,+\right)$ a right zero semigroup, then the map $r$ as in \eqref{eq:r-solution} is a solution if and only if it holds $(a\circ b)^0=(a^0\circ b)^0$, for all $a,b\in S$. 
Observe that semigroups $(S,\circ)$ satisfying such a condition, lie in the wide class of right cryptogroups, see \cite{PeRe99}. 
In this way, we get new idempotent solutions that are of the form $r\left(a, b\right) = \left(a^0, a\circ b\right)$, different from those obtained in \cite{Le17, MaShi18, CvVer20, StVo20x}.\\ 
Moreover, note that if $\left(S,+,\circ\right)$ is the generalized left semi-brace of \cref{Cliff}, we obtain that $r\left(a,b\right) = \left(a^0\circ b, b^-\circ a\circ b\right)$ is a solution.
In particular, if $\left(S,\circ\right)$ is commutative, clearly $r\left(a,b\right) = \left(a^0\circ b, b^0\circ a\right)$ and it is easy to verify that $r$ is a cubic solution, i.e., $r^3 = r$.
\medskip

Our aim is to show that if $S=[Y; S_\alpha,\phii{\alpha}{\beta}]$ is a strong semilattice of generalized left semi-braces such that every $S_\alpha$ satisfies condition \eqref{eq:condsolutionsemi},  then the map 
in \eqref{eq:r-solution} is a solution. 
This result is a consequence of a more general construction technique on solutions we introduce in the following theorem. 
\begin{theor}\label{th:sol-sss}
	Let $Y$ be a (lower) semilattice, 
	$\left\{\X{\alpha}\ \left|\ \alpha\in Y \right.\right\}$ a family of disjoint sets indexed by $Y$, and $r_{\alpha}$ a solution on $\X{\alpha}$, for every $\alpha\in Y$.  
	For each pair $\alpha,\beta$ of elements of $Y$ such that $\alpha\geq \beta$, let $\phii{\alpha}{\beta}:\X{\alpha}\to \X{\beta}$ be a map. 
	If the following conditions are satisfied 
	\begin{enumerate}
		\item $\phii{\alpha}{\alpha}$ is the identity map of $\X{\alpha}$, for every $\alpha\in Y$,
		\item $\phii{\beta}{\gamma}\phii{\alpha}{\beta} = \phii{\alpha}{\gamma}$, for all $\alpha, \beta, \gamma \in Y$ such that $\alpha \geq \beta \geq \gamma$,
		\item 
		$\left(\phii{\alpha}{\beta}\times \phii{\alpha}{\beta}\right)r_{\alpha}
		= r_{\beta}\left(\phii{\alpha}{\beta}\times \phii{\alpha}{\beta}\right)$, for all $\alpha, \beta\in Y$ such that  $\alpha \geq \beta$,
	\end{enumerate}
	set $X = \bigcup\left\{\X{\alpha}\ \left|\ \alpha\in Y\right.\right\}$,  then the map $r:X\times X \longrightarrow X\times X$ defined by 
	\begin{align*}
	r\left(x, y\right):= 
	r_{\alpha\beta}\left(\phii{\alpha}{\alpha\beta}\left(x\right),
	\phii{\beta}{\alpha\beta}\left(y\right) \right),
	\end{align*}	
	for all $x\in X_{\alpha}$ and $y\in X_{\beta}$, is a solution on $X$. 
	We call the pair $\left(X, r\right)$ the \emph{strong semilattice $Y$ of solutions $\left(X_\alpha, r_\alpha\right)$}.
\end{theor}	
	
The proof of \cref{th:sol-sss} is technical, and for the sake of clarity, we present it in the next section.	
Now, as a consequence of this theorem, we obtain the following result.
\begin{theor}
	Let $S=[Y; S_\alpha,\phii{\alpha}{\beta}{}]$ be  a strong semilattice of generalized left semi-braces. Then, if $\s{\alpha}$ satisfies \eqref{eq:condsolutionsemi}, for every $\alpha \in Y$, then the map $r_S:S \times S \to S \times S$ defined by 
	\begin{align*}
	r\left(a,b\right):=\left(a\circ\left(a^-+b\right), \left(a^-+b\right)^-\circ b\right),  
	\end{align*}
	for all $a,b \in S$, is a solution.
\end{theor}
\begin{proof}
	For any $\alpha \in Y$, let $r_{\alpha}:S_{\alpha}\times S_{\alpha}\to S_{\alpha}\times S_{\alpha}$ be the solution associated to the left semi-brace $S_{\alpha}$, i.e., the map defined by $r_{\alpha}\left(x,y\right)=\left(x\circ\left(x^-+y\right),\left(x^-+y\right)^-\circ y\right)$. Since $S$ is a strong semilattice of left semi-braces, by \cref{prop:StrongLatticeGeneralizedSemiBrace}, $\phii{\alpha}{\alpha}{}$ is the identical automorphism of $S_{\alpha}$ and $\phii{\beta}{\gamma}\phii{\alpha}{\beta} = \phii{\alpha}{\gamma}$, for all $\alpha, \beta, \gamma \in Y$ such that $\alpha \geq \beta \geq \gamma$. 
	Hence, conditions $1.$ and $2.$ in \cref{th:sol-sss} are satisfied. Moreover, let $a,b \in Y$ such that $\alpha \geq \beta$. Since, $\phii{\alpha}{\beta}$ is a homomorphism of left semi-braces, for all $x,y\in S_{\alpha}$ it follows that
	\begin{align*}
	&\left(\phii{\alpha}{\beta}\times \phii{\alpha}{\beta}\right)r_{\alpha}\left(x,y\right)=\left(\phii{\alpha}{\beta}\left(x\circ\left(x^-+y\right)\right), \phii{\alpha}{\beta}\left(\left(x^-+y\right)^-\circ y\right)\right)\\
	&=\left(\phii{\alpha}{\beta}\left(x\right)\circ\left(\left(\phii{\alpha}{\beta}\left(x\right)\right)^-+\phii{\alpha}{\beta}\left(y\right)\right), \left(\left(\phii{\alpha}{\beta}\left(x\right)\right)^-+\phii{\alpha}{\beta}\left(y\right)\right)^-\circ\phii{\alpha}{\beta}\left(y\right)\right)\\
	&=r_{\beta}\left(\phii{\alpha}{\beta}\left(x\right),\phii{\alpha}{\beta}\left(y\right)\right)\\
	&=r_{\beta}\left(\phii{\alpha}{\beta}\times \phii{\alpha}{\beta}\right)\left(x,y\right).
	\end{align*}
	Hence $(3)$ in \cref{th:sol-sss} holds. Therefore, according to \cref{th:sol-sss} we shall consider the strong semilattice $Y$ of solutions $r_{\alpha}$, i.e., the map $r$ defined by
	\begin{align*}
	r\left(x,y\right)=r_{\alpha\beta}\left(\phii{\alpha}{\alpha\beta}\left(x\right),\phii{\beta}{\alpha\beta}\left(y\right)\right),
	\end{align*}
	for all $x\in S_{\alpha}$, $y\in S_{\beta}$. Finally, note that, by \cref{prop:StrongLatticeGeneralizedSemiBrace}
	\begin{align*}
	&r\left(x,y\right) \\&\quad= \left(\phii{\alpha}{\alpha\beta}\left(x\right)\circ\left(\left(\phii{\alpha}{\alpha\beta}\left(x\right)\right)^-+\phii{\beta}{\alpha\beta}\left(y\right)\right), \left(\left(\phii{\alpha}{\alpha\beta}\left(x\right)\right)^-+\phii{\beta}{\alpha\beta}\left(y\right)\right)^-\circ \phii{\beta}{\alpha\beta}\left(y\right)\right)\\
	&\quad=\left(x\circ\left(x^-+y\right),\left(x^-+y\right)^-\circ y\right),
	\end{align*}
	for all $x\in S_{\alpha}$, $y\in S_{\beta}$. 
\end{proof}
	
% ---------------------------------------------------
\section{Strong semilattices of set-theoretical solutions}
% ---------------------------------------------------
This section aims to provide a proof of \cref{th:sol-sss} and to give some examples of strong semilattices of solutions. Furthermore, we focus on analyzing strong semilattices of solutions with finite order.
\smallskip

	\begin{proof}[{Proof of \cref{th:sol-sss}}]
	At first, note that if 
	$\lambdai{[\omega]}{x}{}$ and $\rhoi{[\omega]}{y}{}$ are the maps from $X_{\omega}$ into itself that define every solution $r_{\omega}$, i.e.,  $r_{\omega}$ is written as 
	$r_{\omega}\left(x, y\right) = \left(\lambdai{[\omega]}{x}{\left(y\right)}, \rhoi{[\omega]}{y}{\left(x\right)}\right)$, for all $x, y\in X_\omega$, then the condition $3.$ is equivalent to the following equalities
	\begin{align}
	\phii{\omega}{\iota}\lambdai{[\omega]}{x}{\left(y\right)} 
	= \lambdai{[\iota]}{\phii{\omega}{\iota}{\left(x\right)}}{\phii{\omega}{\iota}\left(y\right)}
	\label{eq:phi-lambda}
	\\
	\phii{\omega}{\iota}\rhoi{[\omega]}{y}{\left(x\right)} 
	= \rhoi{[\iota]}{\phii{\omega}{\iota}\left(y\right)}{\phii{\omega}{\iota}\left(x\right)}
	\label{eq:phi-rho}
	\end{align}
	for all $\omega,\iota\in Y$ such that $\omega\geq \iota$ and $x,y\in X_\omega$.
	In addition, let us observe that if $x\in X_\omega$ and $y\in Y_\iota$, then the two components of the map $r$, i.e.,  
	\begin{align*}
		\lambdaa{x}{\left(y\right)}
		= 
		\lambdai{[\omega\iota]}{\phii{\omega}{\omega\iota}\left(x\right)}{\phii{\iota}{\omega\iota}\left(y\right)}
		\qquad 
		\rhoo{y}{\left(x\right)}
		=
		\rhoi{[\omega\iota]}{\phii{\iota}{\omega\iota}\left(y\right)}{\phii{\omega}{\omega\iota}\left(x\right)},
	\end{align*}
	 lie in $X_{\omega\iota}$, consistently with the second part of the subscript of the maps $\phi$.
	To avoid overloading the notation, hereinafter we will write the previous elements as 
	\begin{align*}
		\lambdaa{x}{\left(y\right)}
		= 
		\lambdaa{\phii{\omega}{\omega\iota}\left(x\right)}{\phii{\iota}{\omega\iota}\left(y\right)}
		\qquad
		\rhoo{y}{\left(x\right)}
		=
		\rhoo{\phii{\iota}{\omega\iota}\left(y\right)}{\phii{\omega}{\omega\iota}\left(x\right)}.
	\end{align*}
	Now, we verify that $r$ is a solution proving that the relations 
		\begin{align*}
		& 
		L_{1}:=
		\lambdaa{x}{\lambdaa{y}{\left(z\right)}}
		=
		\lambdaa{\lambdaa{x}{\left(y\right)}}{\lambdaa{\rhoo{y}{\left(x\right)}}{\left(z\right)}}
		=: L_{2} 
		\\
		&C_{1}:=
		\lambdaa{\rhoo{\lambdaa{y}{\left(z\right)}}{\left(x\right)}}{\rhoo{z}{\left(y\right)}} 
		=
		\rhoo{\lambdaa{\rhoo{y}{\left(x\right)}}{\left(z\right)}}{\lambdaa{x}{\left(y\right)}}
		=: C_{2}
		\\
		&R_{2}:=
		\rhoo{\rhoo{y}{\left(z\right)}}{\rhoo{\lambdaa{z}{\left(y\right)}}{\left(x\right)}} 
		= \rhoo{z}{\rhoo{y}{\left(x\right)}}
		=:R_{1}
		\end{align*}
		are satisfied, for all $x, y, z\in X$.
		At this purpose, let $x, y, z$ be elements of $\X{\omega}$, $\X{\iota}$, $\X{\kappa}$ respectively, assume
		$\nu:= \omega\iota\kappa$ and 
		\begin{align*}
		&\mathcal{X}:= \phii{\omega}{\nu}{\left(x\right)}
		\qquad
		\mathcal{Y}:= \phii{\iota}{\nu}{\left(y\right)}
		\qquad
		\mathcal{Z}:= \phii{\kappa}{\nu}{\left(z\right)}
		&\mbox{in $\X{\nu}$}.
		\end{align*}
		Setting 
		\begin{align*}
		&U:= \rhoo{\phii{\iota}{\omega\iota}\left(y\right)}{\phii{\omega}{\omega\iota}\left(x\right)}
		\qquad
		V:= \lambdaa{\phii{\omega}{\omega\iota}\left(x\right)}{\phii{\iota}{\omega\iota}\left(y\right)}
		&\mbox{in $\X{\omega\iota}$}
		\end{align*}
		we have that  
		\begin{align*}
		L_{2}
		&= \lambdaa{\lambdaa{x}{\left(y\right)}}{\lambdaa{\rhoo{y}{\left(x\right)}}{\left(z\right)}}
		=
		\lambdaa{V}{\lambdaa{U}{\left(z\right)}}\\
		&= \lambdaa{V}
		{\lambdaa{\phii{\omega\iota}{\nu}\left(U\right)}
			{\phii{\kappa}{\nu}\left(z\right)}}
		&\mbox{$\nu = \omega\iota\kappa$}\\
		&= \lambdaa
		{\phii{\omega\iota}{\omega\iota\nu}\left(V\right)}
		{\phii{\nu}{\omega\iota\nu}
			\lambdaa{\phii{\omega\iota}{\nu}\left(U\right)}
			{\left(\mathcal{Z}\right)}}\\
		&= \lambdaa
		{\phii{\omega\iota}{\nu}\left(V\right)}
		{\phii{\nu}{\nu}
			\lambdaa{\phii{\omega\iota}{\nu}\left(U\right)}
			{\left(\mathcal{Z}\right)}}
		&\mbox{$\omega\iota\nu = \nu$}\\
		&= \lambdaa
		{\phii{\omega\iota}{\nu}\left(V\right)}
		{\lambdaa{\phii{\omega\iota}{\nu}\left(U\right)}
			{\left(\mathcal{Z}\right)}}.
		&\mbox{$\phii{\nu}{\nu} = \id_{\X{\nu}}$}
		\end{align*}
		Since
		\begin{align*}
		\phii{\omega\iota}{\nu}
		&\left(V\right)
		= 
		\lambdaa{\phii{\omega\iota}{\nu}\phii{\omega}{\omega\iota}\left(x\right)}
		{\phii{\omega\iota}{\nu}\phii{\iota}{\omega\iota}\left(y\right)}
		&\mbox{by \eqref{eq:phi-lambda}}\\
		&= \lambdaa{\phii{\omega}{\nu}\left(x\right)}{\phii{\iota}{\nu}\left(y\right)}
		&\mbox{$\phii{\omega\iota}{\nu}\phii{\omega}{\omega\iota} 
			= \phii{\omega}{\nu}$, 
			$\phii{\omega\iota}{\nu}\phii{\iota}{\omega\iota} = \phii{\iota}{\nu}$
		}\\
		&= \lambdaa{\mathcal{X}}{\left(\mathcal{Y}\right)}
		\end{align*}
		and
		\begin{align*}
		\phii{\omega\iota}{\nu}
		&\left(U\right)
		= \rhoo{\phii{\omega\iota}{\nu}\phii{\iota}{\omega\iota}\left(y\right)}{\phii{\omega\iota}{\nu}\phii{\omega}{\omega\iota}\left(x\right)}
		&\mbox{by \eqref{eq:phi-rho}}\\
		&= \rhoo{\phii{\iota}{\nu}\left(y\right)}
		{\phii{\omega}{\nu}\left(x\right)}
		&\mbox{$\phii{\omega\iota}{\nu}\phii{\iota}{\omega\iota} 
			= \phii{\iota}{\nu}$, 
			$\phii{\omega\iota}{\nu}\phii{\omega}{\omega\iota} = \phii{\omega}{\nu}$
		}\\
		&= \rhoo{\mathcal{Y}}{\left(\mathcal{X}\right)},
		\end{align*}
		it follows that 
		\begin{align*}
		L_{2} &= 
		\lambdaa{\lambdaa{\mathcal{X}}{\left(\mathcal{Y}\right)}}
		{\lambdaa
			{\rhoo{\mathcal{Y}}{\left(\mathcal{X}\right)}}
			{\left(\mathcal{Z}\right)}}
		&\mbox{in $\X{\nu}$}\\
		&=  \lambdaa{\mathcal{X}}{\lambdaa{\mathcal{Y}}{\left(\mathcal{Z}\right)}}.
		&\mbox{$r_{\nu}$ is a solution}
		\end{align*}
		Moreover, it holds
		\begin{align*}
		L_{1} &=\lambdaa{x}{\lambdaa{y}{\left(z\right)}}
		=\lambdaa{x}
		{\lambdaa{\phii{\iota}{\iota\kappa}\left(y\right)}
			{\phii{\kappa}{\iota\kappa}\left(z\right)}}\\
		&=\lambdaa{\phii{\omega}{\nu}\left(x\right)}
		{\phii{\iota\kappa}{\nu}\lambdaa{\phii{\iota}{\iota\kappa}\left(y\right)}{\phii{\kappa}{\iota\kappa}\left(z\right)}}
		&\mbox{$\nu=\omega\iota\kappa$}\\
		&= \lambdaa
		{\mathcal{X}}
		{\lambdaa{\phii{\iota\kappa}{\nu}\phii{\iota}{\iota\kappa}\left(y\right)}
			{\phii{\iota\kappa}{\nu}\phii{\kappa}{\iota\kappa}\left(z\right)}}
		&\mbox{by \eqref{eq:phi-lambda}}\\
		&= \lambdaa
		{\mathcal{X}}
		{\lambdaa{\phii{\iota}{\nu}\left(y\right)}
			{\phii{\kappa}{\nu}\left(z\right)}}
		&\mbox{$\phii{\iota\kappa}{\nu}\phii{\iota}{\iota\kappa} 
			= \phii{\iota}{\nu}$, 
			$\phii{\iota\kappa}{\nu}\phii{\kappa}{\iota\kappa} = \phii{\kappa}{\nu}$}\\
		&= \lambdaa
		{\mathcal{X}}{\lambdaa{\mathcal{Y}}{\left(\mathcal{Z}\right)}},
		\end{align*}
		hence we obtain that $L_{1} = L_{2}$.
		Now, setting
		\begin{align*}
		&W:= \rhoo{\phii{\kappa}{\iota\kappa}\left(z\right)}{\phii{\iota}{\iota\kappa}\left(y\right)}
		\qquad
		Z:= \lambdaa{\phii{\iota}{\iota\kappa}\left(y\right)}{\phii{\kappa}{\iota\kappa}\left(z\right)}
		&\mbox{in $\X{\iota\kappa}$}
		\end{align*}
		observe that 
		\begin{align*}
		C_{1} 
		&= \lambdaa{\rhoo{\lambdaa{y}{\left(z\right)}}{\left(x\right)}}{\rhoo{z}{\left(y\right)}}\\
		&=\lambdaa{\rhoo{\lambdaa{\phii{\iota}{\iota\kappa}\left(y\right)}{\phii{\kappa}{\iota\kappa}\left(z\right)}}{\left(x\right)}}
		{\rhoo{\phii{\kappa}{\iota\kappa}\left(z\right)}{\phii{\iota}{\iota\kappa}\left(y\right)}}\\
		&= \lambdaa{\rhoo{Z}{\left(x\right)}}
		{\left(W\right)}\\
		&= \lambdaa{\rhoo{\phii{\iota\kappa}{\nu}\left(Z\right)}
			{\phii{\omega}{\nu}}\left(x\right)}
		{\left(W\right)}
		&\mbox{$\nu =\omega\iota\kappa$}\\
		&= 
		\lambdaa{\phii{\nu}
			{\nu\iota\kappa}
			\rhoo{\phii{\iota\kappa}{\nu}\left(Z\right)}
			{\left(\mathcal{X}\right)}}
		{\phii{\iota\kappa}{\nu\iota\kappa}\left(W\right)}
		\\
		&= \lambdaa{\phii{\nu}{\nu}
			\rhoo{\phii{\iota\kappa}{\nu}\left(Z\right)}
			{\left(\mathcal{X}\right)}}
		{\phii{\iota\kappa}{\nu}\left(W\right)} 
		&\mbox{$\nu\iota\kappa = \nu$}\\
		&= \lambdaa{\rhoo{\phii{\iota\kappa}{\nu}\left(Z\right)}
			{\left(\mathcal{X}\right)}}
		{\phii{\iota\kappa}{\nu}\left(W\right)}.
		&\mbox{$\phii{\nu}{\nu} = \id_{\X{\nu}}$}
		\end{align*}
		Since
		\begin{align*}
		\phii{\iota\kappa}{\nu}&\left(Z\right)
		= \lambdaa{\phii{\iota\kappa}{\nu}\phii{\iota}{\iota\kappa}
			\left(y\right)}{\phii{\iota\kappa}{\nu}\phii{\kappa}{\iota\kappa}\left(z\right)}
		&\mbox{by \eqref{eq:phi-lambda}}\\
		&= \lambdaa{\phii{\iota}{\nu}
			\left(y\right)}{\phii{\kappa}{\nu}\left(z\right)}
		&\mbox{$\phii{\iota\kappa}{\nu}\phii{\iota}{\iota\kappa} = \phii{\iota}{\nu}$, 
			$\phii{\iota\kappa}{\nu}\phii{\kappa}{\iota\kappa} = \phii{\kappa}{\nu}$}\\
		&= \lambdaa{\mathcal{Y}}{\left(\mathcal{Z}\right)}
		\end{align*}
		and
		\begin{align*}
		\phii{\iota\kappa}{\nu}
		&\left(W\right)
		=\rhoo{\phii{\iota\kappa}{\nu}\phii{\kappa}{\iota\kappa}\left(z\right)}
		{\phii{\iota\kappa}{\nu}\phii{\iota}{\iota\kappa}\left(y\right)} 
		&\mbox{by \eqref{eq:phi-rho}}\\
		&=\rhoo{\phii{\kappa}{\nu}\left(z\right)}
		{\phii{\iota}{\nu}\left(y\right)} 
		&\mbox{$\phii{\iota\kappa}{\nu}\phii{\kappa}{\iota\kappa} =\phii{\kappa}{\nu}$,
			$\phii{\iota\kappa}{\nu}\phii{\iota}{\iota\kappa} =\phii{\iota}{\nu}$}\\
		&=\rhoo{\mathcal{Z}}
		{\left(\mathcal{Y}\right)} 	
		\end{align*}
		it follows that
		\begin{align*}
		C_{1} &= 
		\lambdaa{\rhoo{\lambdaa{\mathcal{Y}}{\left(\mathcal{Z}\right)}}{\left(\mathcal{X}\right)}}{\rhoo{\mathcal{Z}}
			{\left(\mathcal{Y}\right)}}.
		&\mbox{in $\X{\nu}$}
		\end{align*}
		Furthermore, since $U$ and $V$ lie in $\X{\omega\iota}$ we have
		\begin{align*}
		C_{2}
		&=\rhoo{\lambdaa{\rhoo{y}{\left(x\right)}}{\left(z\right)}}{\lambdaa{x}{\left(y\right)}}\\
		&=\rhoo{\lambdaa{\rhoo{\phii{\iota}{\omega\iota}\left(y\right)}{\phii{\omega}{\omega\iota}\left(x\right)}}{\left(z\right)}}{\lambdaa{\phii{\omega}{\omega\iota}\left(x\right)}{\phii{\iota}{\omega\iota}\left(y\right)}}\\
		&=
		\rhoo{\lambdaa{U}{\left(z\right)}}{\left(V\right)}
		\\
		&= \rhoo{\lambdaa{\phii{\omega\iota}{\nu}\left(U\right)}
			{\phii{\kappa}{\nu}\left(z\right)}}
		{\left(V\right)}
		&\mbox{$\nu = \omega\iota\kappa$}\\
		&= \rhoo{\phii{\nu}{\omega\iota\nu}
			\lambdaa{\phii{\omega\iota}{\nu}\left(U\right)}
			{\left(\mathcal{Z}\right)}}{
			\phii{\omega\iota}{\omega\iota\nu}\left(V\right)}\\
		&= \rhoo{\phii{\nu}{\nu}
			\lambdaa{\phii{\omega\iota}{\nu}\left(U\right)}
			{\left(\mathcal{Z}\right)}}{
			\phii{\omega\iota}{\nu}\left(V\right)}
		&\mbox{$\omega\iota\nu = \nu$}\\
		&= \rhoo{
			\lambdaa{\phii{\omega\iota}{\nu}\left(U\right)}
			{\left(\mathcal{Z}\right)}}{
			\phii{\omega\iota}{\nu}\left(V\right)}.
		&\mbox{$\phii{\nu}{\nu} = \id_{\X{\nu}}$}
		\end{align*}
		As seen before, $\phii{\omega\iota}{\nu}\left(U\right)
		= 
		\rhoo{\mathcal{Y}}
		{\left(\mathcal{X}\right)}$
		and
		$\phii{\omega\iota}{\nu}\left(V\right)
		=
		\lambdaa{\mathcal{X}}{\left(\mathcal{Y}\right)}$,
		thus 
		\begin{align*}
			C_{2} &= \rhoo{\lambdaa{\rhoo{\mathcal{Y}}
				{\left(\mathcal{X}\right)}}{\left(\mathcal{Z}\right)}}
			{\lambdaa{\mathcal{X}}{\left(\mathcal{Y}\right)}}.
			&\mbox{in $\X{\nu}$}
		\end{align*}
		Consequently, since $r_{\nu}$ is a solution, we obtain that $C_{1} = C_{2}$.\\
		Finally, since $W$ and $Z$ lie in $\X{\iota\kappa}$, note that
		\begin{align*}
		R_{2}
		&=\rhoo{\rhoo{y}{\left(z\right)}}{\rhoo{\lambdaa{z}{\left(y\right)}}{\left(x\right)}}\\ 
		&=\rhoo
		{\rhoo{\phii{\kappa}{\iota\kappa}\left(z\right)}
			{\phii{\iota}{\iota\kappa}\left(y\right)}}
		{\rhoo{\lambdaa{\phii{\iota}{\iota\kappa}\left(y\right)}
				{\phii{\kappa}{\iota\kappa}\left(z\right)}}
			{\left(x\right)}}\\
		&= \rhoo{W}{\rhoo{Z}{\left(x\right)}}\\
		&=
		\rhoo{W}
		{\rhoo{\phii{\iota\kappa}{\nu}
				\left(Z\right)}
			{\phii{\omega}{\nu}\left(x\right)}}
		&\mbox{$\nu = \omega\iota\kappa$}\\
		&= \rhoo
		{\phii{\iota\kappa}{\nu\iota\kappa}\left(W\right)}
		{\phii{\nu}{\nu\iota\kappa}
			\rhoo{\phii{\iota\kappa}{\nu}\left(Z\right)}
			{\left(\mathcal{X}\right)}}\\
		&= \rhoo
		{\phii{\iota\kappa}{\nu}\left(W\right)}
		{\phii{\nu}{\nu}
			\rhoo{\phii{\iota\kappa}{\nu}\left(Z\right)}
			{\left(\mathcal{X}\right)}}
		&\mbox{$\nu\iota\kappa = \nu$}\\
		&= \rhoo
		{\phii{\iota\kappa}{\nu}\left(W\right)}
		{\rhoo{\phii{\iota\kappa}{\nu}\left(Z\right)}
			{\left(\mathcal{X}\right)}}
		&\mbox{$\phii{\nu}{\nu} = \id_{\X{\nu}}$}
		\end{align*}
		As seen before, $\phii{\iota\kappa}{\nu}\left(W\right)
		= \rhoo{\mathcal{Z}}
		{\left(\mathcal{Y}\right)}$.
		In addition we have 
		\begin{align*}
		\phii{\iota\kappa}{\nu}
		&\left(Z\right)
		=\lambdaa{\phii{\iota\kappa}{\nu}\phii{\iota}{\iota\kappa}\left(y\right)}
		{\phii{\iota\kappa}{\nu}\phii{\kappa}{\iota\kappa}\left(z\right)}
		&\mbox{by \eqref{eq:phi-lambda}}\\
		&= \lambdaa{\phii{\iota}{\nu}\left(y\right)}
		{\phii{\kappa}{\nu}\left(z\right)}
		&\mbox{$\phii{\iota\kappa}{\nu}\phii{\iota}{\iota\kappa} = 			\phii{\iota}{\nu}$, 
			$\phii{\iota\kappa}{\nu}\phii{\kappa}{\iota\kappa} = \phii{\kappa}{\nu}$}\\
		&= \lambdaa{\mathcal{Y}}{\left(\mathcal{Z}\right)}.
		\end{align*}
		It follows that  
		\begin{align*}
		R_{2} &= \rhoo{\rhoo{\mathcal{Z}}
			{\left(\mathcal{Y}\right)}}
		{\rhoo{\lambdaa{\mathcal{Y}}
				{\left(\mathcal{Z}\right)}}
			{\left(\mathcal{X}\right)}}
		&\mbox{in $\X{\nu}$}\\
		&= \rhoo{\mathcal{Z}}{\rhoo{\mathcal{Y}}{\left(\mathcal{X}\right)}}.
		&\mbox{$r_{\nu}$ is a solution}
		\end{align*}
		Moreover, it holds 
		\begin{align*}
		R_{1} 
		&= \rhoo{z}{\rhoo{y}{\left(x\right)}}
		=\rhoo{z}{\rhoo{\phii{\iota}{\omega\iota}\left(y\right)}{\phii{\omega}{\omega\iota}\left(x\right)}}\\
		&=
		\rhoo{\phii{\kappa}{\nu}\left(z\right)}
		{\phii{\omega\iota}{\nu}\rhoo{\phii{\iota}{\omega\iota}\left(y\right)}{\phii{\omega}{\omega\iota}\left(x\right)}}\\
		&= \rhoo{\mathcal{Z}}
		{\rhoo{\phii{\omega\iota}{\nu}\phii{\iota}{\omega\iota}\left(y\right)}{\phii{\omega\iota}{\nu}\phii{\omega}{\omega\iota}\left(x\right)}}
		&\mbox{by \eqref{eq:phi-rho}}\\
		&= \rhoo{\mathcal{Z}}
		{\rhoo{\phii{\iota}{\nu}\left(y\right)}
			{\phii{\omega}{\nu}\left(x\right)}}
		&\mbox{$\phii{\omega\iota}{\nu}\phii{\iota}{\omega\iota} = \phii{\iota}{\nu}$,
			$\phii{\omega\iota}{\nu}\phii{\omega}{\omega\iota} = \phii{\omega}{\nu}$}\\
		&= \rhoo{\mathcal{Z}}
		{\rhoo{\mathcal{Y}}
			{\left(\mathcal{X}\right)}}, 
		\end{align*}
		hence $R_{1} = R_{2}$. Therefore, the map $r$ is a solution.
	\end{proof}
\medskip
 	
 Strong semilattices of solutions $\left(X_\alpha, r_{\alpha}\right)$ allows one to produce examples of solutions with finite order if solutions $r_{\alpha}$ are.   
	\begin{exs}
		$1.$ \ Let $X$ be a semilattice of sets indexed by $Y$ such that conditions $1.$ and $2.$ of \cref{th:sol-sss} are satisfied and let $r_{\alpha}$ be the twist map on $\X{\alpha}$, for every $\alpha\in Y$. Then, if  $\alpha\geq\beta$ we have that
		\begin{align*}
			&\phii{\alpha}{\beta}\lambdaa{x}{\left(y\right)} 
			= \phii{\alpha}{\beta}\left(y\right) 
			= \lambdaa{\phii{\alpha}{\beta}{\left(x\right)}}{\phii{\alpha}{\beta}\left(y\right)}
			\\
			&\phii{\alpha}{\beta}\rhoo{y}{\left(x\right)} 
			= \phii{\alpha}{\beta}\left(x\right)
			= \rhoo{\phii{\alpha}{\beta}\left(y\right)}{\phii{\alpha}{\beta}\left(x\right)},
		\end{align*}
		for all $x, y\in X_\alpha$, i.e.,  condition $3.$ of \cref{th:sol-sss} holds. 
		Hence, the strong semilattice of solutions $\left(X, r\right)$ is such that $r^{3} = r$. 
		Indeed, if $x\in X_\alpha$ and $y\in X_\beta$, assuming $\nu:=\alpha\beta$ we have that 
		$r\left(x, y\right) 
		= r_{\nu}\left(\phii{\alpha}{\nu}\left(x\right),
		\phii{\beta}{\nu}\left(y\right)\right)$, hence 
		$r^{2}\left(x, y\right)
		= \left(\phii{\alpha}{\nu}\left(x\right), \phii{\beta}{\nu}\left(y\right)\right)$
		and so 
		\begin{align*}
		r^{3}\left(x, y\right)
		= r_{\nu}\left(\phii{\alpha}{\nu}\left(x\right), \phii{\beta}{\nu}\left(y\right)\right)
		=
		r\left(x, y\right).
		\end{align*}
		Consequently, $r^{3} = r$.
		
		$2.$ \  Let $X:= \X{\alpha}\cup \X{\beta}$ be a semilattice of sets
		such that $\alpha > \beta$, $c$ a fixed element of $X_{\beta}$,
		and $\phii{\alpha}{\beta}\left(x\right) = c$, for every $x\in X_{\alpha}$.
		Let $r_{\alpha}$ be the twist map on $\X{\alpha}$ and $r_{\beta}$ the idempotent solution on $\X{\beta}$ defined by $r_{\beta}\left(x, y\right):= \left(x, c\right)$, for all $x,y\in X_\beta$.  
		Then, if $x,y\in X_\alpha$, we have that 
		\begin{align*}
		\phii{\alpha}{\beta}\lambdaa{x}{\left(y\right)} 
		&= c
		= \phii{\alpha}{\beta}\left(x\right) =\lambdaa{\phii{\alpha}{\beta}{\left(x\right)}}{\phii{\alpha}{\beta}\left(y\right)}
		\\
		\phii{\alpha}{\beta}\rhoo{y}{\left(x\right)}
		&= c
		=   \rhoo{\phii{\alpha}{\beta}\left(y\right)}{\phii{\alpha}{\beta}\left(x\right)},
		\end{align*}
		hence the assumptions of \cref{th:sol-sss} are satisfied.
		Moreover, the strong semilattice of solutions $\left(X, r\right)$ is such that $r^{3} = r$.
		Indeed, if $x\in X_{\alpha}$ and  $y\in X_{\beta}$, since 
		$r\left(x, y\right) 
		= r_{\beta}\left(\phii{\alpha}{\beta}\left(x\right),
		y\right)
		= r_{\beta}\left(c, y\right) = \left(c, c\right)$, 
		it follows that 
		\begin{align*}
			r^{2}\left(x, y\right)
			= r_{\beta}\left(c, c\right) 
			= \left(c, c\right) 
			= r\left(x, y\right).
		\end{align*}
		Therefore, we obtain that $r^{3} = r$.
		
		$3.$ \ Let $X:= \X{\alpha}\cup \X{\beta}$ be a semilattice of sets 
		such that $\alpha > \beta$, $c$ a fixed element of $X_{\beta}$, and $\phii{\alpha}{\beta}\left(x\right) = c$, for every $x\in X_{\alpha}$.
		Let $f$ be an idempotent map from $\X{\beta}$ into itself, $f\neq \id_{\X{\beta}}$, and 
		$r_{\alpha}$ the map from $\X{\alpha}\times \X{\alpha}$ into itself defined by $r_{\alpha}\left(x, y\right):= \left(f\left(x\right), x\right)$, for all $x, y\in X_{\alpha}$. 
		Thus, $r_{\alpha}$ is a solution such that $r_{\alpha}^{3} = r_{\alpha}^{2}$.
		Let $r_{\beta}$ be the idempotent solution defined by $r_{\beta}\left(x, y\right) = \left(x, c\right)$, for all $x, y\in X_\beta$. 
		Then, if $x,y\in X_\alpha$, we obtain that
		\begin{align*}
		\phii{\alpha}{\beta}\lambdaa{x}{\left(y\right)} 
		&= c
		= \phii{\alpha}{\beta}{\left(x\right)} = \lambdaa{\phii{\alpha}{\beta}{\left(x\right)}}{\phii{\alpha}{\beta}\left(y\right)}
		\\
		\phii{\alpha}{\beta}\rhoo{y}{\left(x\right)} 
		&= c
		=   \rhoo{\phii{\alpha}{\beta}\left(y\right)}{\phii{\alpha}{\beta}\left(x\right)},
		\end{align*}
		thus the hypotheses of \cref*{th:sol-sss} are satisfied.
		Moreover, the strong semilattice of solutions $\left(X, r\right)$ is such that $r^{3} = r^{2}$. 
		Indeed, if $x\in X_\alpha$ and $y\in X_\beta$, since
		$r\left(x, y\right) 
		= r_{\beta}\left(\phii{\alpha}{\beta}\left(x\right),
		y\right)
		= r_{\beta}\left(c, y\right) = \left(c,c\right)$, 
		it follows that 
		\begin{align*}
			r^{2}\left(x, y\right)
			= r_{\beta}\left(c, c\right)
			= \left(c, c\right)
			= r\left(x, y\right)
		\end{align*}
		and clearly $r^{3}\left(x, y\right) = r^{2}\left(x, y\right)$.
		Therefore, $r^{3} = r^{2}$. 
   \end{exs}
   \medskip
 	
 	To investigate strong semilattices of solutions with finite order, we need the notions of the index and the period of a solution $r$ that are 
 	\begin{align*}
 	\indd{\left(r\right)}
 	&:=\min\left\{\left.j \,\right|\, j\in\mathbb{N}_0, \, \exists \, l\in \mathbb{N}\ r^j = r^l , \ j\neq l\right\}\\
 	\perr{\left(r\right)}
 	&:=\min\left\{\left.k\,\right| \, k\in\mathbb{N}, \, r^{\indd{\left(r\right)}+k} = r^{\indd{\left(r\right)}}\right\},
 	\end{align*}
 	respectively. 
    These definitions of the index and the order are slightly different from the classical ones (cf. \cite[p. 10]{Ho95book}), but they are functional to distinguish bijective solutions from non-bijective ones. For more details, we refer the reader to \cite{CCoSt20-2}.
 	
 	In the following theorem we show that, given a semilattice $Y$ of finite cardinality, the strong semilattice of solutions $\left(X, r\right)$ indexed by $Y$ is of finite order if and only if solutions $r_{\alpha}$ are.
 	Furthermore, it allows for establishing the order of the strong semilattice of solutions $\left(X, r\right)$ if the index and the period of solutions $r_{\alpha}$ are known. 
 	Conversely, the index and the period of a strong semilattice of solutions $\left(X, r\right)$ give us upper bounds of the indexes ad periods of solutions $r_{\alpha}$.
	\begin{theor}\label{th:sol-sss-finite}
		Let $\left(X, r\right)$ be a strong semilattice of solutions indexed by a finite semilattice $Y$. 
		Then, $r_{\alpha}$ is a solution with finite order on $\X{\alpha}$, for every $\alpha\in Y$, if and only if $r$ is a solution with finite order.
		More precisely, the index of $r$ is
\begin{align*}
    \indd{\left(r\right)}
    = \max\left\lbrace  \ 1, \  \indd{\left(r_{\alpha}\right)} \ | \ \alpha\in Y \ \right\rbrace
\end{align*}
		and the period is 
		\begin{align*}
			\perr{\left(r\right)} = \lcm\left\lbrace \perr{\left(r_{\alpha}\right)} \ | \  \alpha\in Y\right\rbrace.
		\end{align*}
	\begin{proof}
	At first suppose that $r_{\alpha}$ is a solution with finite order, for every $\alpha\in Y$.  
	Let $n:= \lcm\left\lbrace \perr{\left(r_{\alpha}\right)} \ | \ \alpha\in Y \right\rbrace$ and $i:= \max\left\lbrace \indd{\left(r_{\alpha}\right)} \ | \ \alpha\in Y \right\rbrace$.
   	If $x\in X_\alpha$ and $y\in X_\beta$, setting $\nu:= \alpha\beta$ we have that 
   	\begin{align*}
   		r^{n + i}\left(x, y\right)
   		= r^{n+i}_{\nu}\left(\phii{\alpha}{\nu}\left(x\right),
   			\phii{\beta}{\nu}\left(y\right)\right)
   		= r^{i}_{\nu}\left(\phii{\alpha}{\nu}\left(x\right),
   		    \phii{\beta}{\nu}\left(y\right)\right).
   \end{align*} 
   Consequently, if $i = 0$, i.e., $r_{\alpha}$ is bijective for every $\alpha\in Y$, we obtain that 
   \begin{align*}
   r^{n + 1}\left(x, y\right)
   = r_{\nu}\left(\phii{\alpha}{\nu}\left(x\right),
   		\phii{\beta}{\nu}\left(y\right)\right)
   = r\left(x, y\right).
   \end{align*}
   Therefore, $r^{n+1} = r$ and clearly the index of $r$ is $1$ by the assumption on $i$. 
   Now, assume that $i\neq 0$
   and that $i = \indd{\left(r_{\gamma}\right)}$, for a certain $\gamma\in Y$.
   If $r^{n} = r^{h}$, for a positive integer $h$, 
   in particular we have that $r_{\gamma}^{n} = r_{\gamma}^{h}$.
   Since $n = \perr{\left(r_{\gamma}\right)}q + i$ for a certain natural number $q$, it follows that 
   \begin{align*}
   r_{\gamma}^{ \ \perr{\left(r_{\gamma}\right)} + i} 
   = r_{\gamma}^{i} 
   = r_{\gamma}^{ \ \perr{\left(r_{\gamma}\right)}q + i}
   = r_{\gamma}^{n} 
   = r_{\gamma}^{h},
   \end{align*}
   hence $i\leq h$ and so $\indd{\left(r\right)} = i$. \\
  	Now, we proceed to determine the period of $r$.
  	If $r^{m} = r^{\indd{\left(r\right)}}$ for a natural number $m$, 
    then $r_{\alpha}^{m} = r_{\alpha}^{\indd{\left(r\right)}}$, for every $\alpha\in Y$. Consequently, $\perr{\left(r_{\alpha}\right)}$ divides $m - \indd{\left(r\right)}$, for every $\alpha \in Y$, thus $\lcm\left\lbrace \perr{\left(r_{\alpha}\right)} \ | \ \alpha\in Y\right\rbrace$ divides $m - \indd{\left(r\right)}$, i.e., $n - \indd{\left(r\right)}$ divides $m - \indd{\left(r\right)}$. 
    Therefore, $n - \indd{\left(r\right)}\leq m - \indd{\left(r\right)}$ and hence $\perr{\left(r\right)} = n - \indd{\left(r\right)}$.
	Conversely, suppose that the solution $r$ is with finite order, set $i:= \indd{\left(r\right)}$, and $p:= \perr{\left(r\right)}$. If $\alpha\in Y$,  since $\phii{\alpha}{\alpha} = \id_{\X{\alpha}\times \X{\alpha}}$ we obtain that $r_{\alpha}^{p + i} = r_{\alpha}^{i}$. Therefore, $r_{\alpha}$ is a solution with finite order. 
	Clearly, we have that $\indd{\left(r_{\alpha}\right)}$ is less than $i$ and $\perr{\left(r_{\alpha}\right)}$ divides $p$.
 	\end{proof}
\end{theor}

Let us note that if $\left(X, r\right)$ is a strong semilattice of non-bijective solutions $r_\alpha$ such that $\indd{\left(r_\alpha\right)} = i$ and $\perr{\left(r_\alpha\right)} = n$, for every $\alpha\in Y$, then $r$ is still a solution of index $i$ and period $n$, also in the case of an infinite semilattice $Y$. Indeed, one can prove this statement by similar computations used for the non-bijective case in the proof of \cref{th:sol-sss-finite}.

\bibliographystyle{elsart-num-sort}  
%\bibliography{references} 
\bibliography{bibliography}

\def\cprime{$'$}
\begin{thebibliography}{10}
\expandafter\ifx\csname url\endcsname\relax
  \def\url#1{\texttt{#1}}\fi
\expandafter\ifx\csname urlprefix\endcsname\relax\def\urlprefix{URL }\fi

\bibitem{Au74}
J.~E. Ault, Semigroups with midunits, Trans. Amer. Math. Soc. 190 (1974)
  375--384.
\newline\urlprefix\url{https://doi.org/10.2307/1996970}

\bibitem{Ba72}
R.~J. Baxter, Partition function of the eight-vertex lattice model, Ann.
  Physics 70 (1972) 193--228.
\newline\urlprefix\url{https://doi.org/10.1016/0003-4916(72)90335-1}

\bibitem{Br18}
T.~Brzezi\'{n}ski, Towards semi-trusses, Rev. Roumaine Math. Pures Appl. 63~(2)
  (2018) 75--89.

\bibitem{Br19}
T.~Brzezi\'{n}ski, Trusses: between braces and rings, Trans. Amer. Math. Soc.
  372~(6) (2019) 4149--4176.
\newline\urlprefix\url{https://doi.org/10.1090/tran/7705}

\bibitem{CCoSt15}
F.~Catino, I.~Colazzo, P.~Stefanelli, On regular subgroups of the affine group,
  Bull. Aust. Math. Soc. 91~(1) (2015) 76--85.
\newline\urlprefix\url{http://dx.doi.org/10.1017/S000497271400077X}

\bibitem{CCoSt16}
F.~Catino, I.~Colazzo, P.~Stefanelli, Regular subgroups of the affine group and
  asymmetric product of radical braces, J. Algebra 455 (2016) 164--182.
\newline\urlprefix\url{http://dx.doi.org/10.1016/j.jalgebra.2016.01.038}

\bibitem{CCoSt17}
F.~Catino, I.~Colazzo, P.~Stefanelli, Semi-braces and the {Y}ang-{B}axter
  equation, J. Algebra 483 (2017) 163--187.
\newline\urlprefix\url{https://doi.org/10.1016/j.jalgebra.2017.03.035}

\bibitem{CCoSt19}
F.~Catino, I.~Colazzo, P.~Stefanelli, Skew left braces with non-trivial
  annihilator, J. Algebra Appl.
\newline\urlprefix\url{https://doi.org/10.1142/S0219498819500336}

\bibitem{CCoSt20}
F.~Catino, I.~Colazzo, P.~Stefanelli, The matched product of set-theoretical
  solutions of the {Y}ang-{B}axter equation, J. Pure Appl. Algebra 224~(3)
  (2020) 1173--1194.
\newline\urlprefix\url{https://doi.org/10.1016/j.jpaa.2019.07.012}

\bibitem{CCoSt20-2}
F.~Catino, I.~Colazzo, P.~Stefanelli, The {M}atched {P}roduct of the
  {S}olutions to the {Y}ang--{B}axter {E}quation of {F}inite {O}rder, Mediterr.
  J. Math. 17, 58 (2020).
\newline\urlprefix\url{https://doi.org/10.1007/s00009-020-1483-y}

\bibitem{CaMaSt20x}
F.~Catino, M.~Mazzotta, P.~Stefanelli, Set-theoretical solutions of the
  {Y}ang-{B}axter and pentagon equations on semigroups, Accepted on Semigroup
  Forum, DOI: 10.1007/s00233-020-10100-x.

\bibitem{CeJO14}
F.~Ced{\'o}, E.~Jespers, J.~Okni{\'n}ski, Braces and the {Y}ang-{B}axter
  equation, Comm. Math. Phys. 327~(1) (2014) 101--116.
\newline\urlprefix\url{http://dx.doi.org/10.1007/s00220-014-1935-y}

\bibitem{Ch18}
L.~N. Childs, Skew braces and the {G}alois correspondence for {H}opf {G}alois
  structures, J. Algebra 511 (2018) 270--291.
\newline\urlprefix\url{https://doi.org/10.1016/j.jalgebra.2018.06.023}

\bibitem{Ch16}
F.~Chouraqui, Left orders in {G}arside groups, Internat. J. Algebra Comput.
  26~(7) (2016) 1349--1359.
\newline\urlprefix\url{https://doi.org/10.1142/S0218196716500570}

\bibitem{ChGo14}
F.~Chouraqui, E.~Godelle, Finite quotients of groups of {I}-type, Adv. Math.
  258 (2014) 46--68.
\newline\urlprefix\url{https://doi.org/10.1016/j.aim.2014.02.009}

\bibitem{Cl41}
A.~H. Clifford, Semigroups admitting relative inverses, Ann. of Math. (2) 42
  (1941) 1037--1049.
\newline\urlprefix\url{https://doi.org/10.2307/1968781}

\bibitem{ClPr61}
A.~H. Clifford, G.~B. Preston, The algebraic theory of semigroups. {V}ol. {I},
  Mathematical Surveys, No. 7, American Mathematical Society, Providence, R.I.,
  1961.

\bibitem{CVA20x}
I.~Colazzo, A.~V. Antwerpen, The algebraic structure of left semi-trusses
  (2019).

\bibitem{CvVer20}
K.~Cvetko-Vah, C.~Verwimp, Skew lattices and set-theoretic solutions of the
  {Y}ang-{B}axter equation, J. Algebra 542 (2020) 65--92.
\newline\urlprefix\url{https://doi.org/10.1016/j.jalgebra.2019.10.007}

\bibitem{DeC19}
K.~De~Commer, Actions of skew braces and set-theoretic solutions of the
  reflection equation, Proc. Edinb. Math. Soc. (2) 62~(4) (2019) 1089--1113.
\newline\urlprefix\url{https://doi.org/10.1017/s0013091519000129}

\bibitem{De15}
P.~Dehornoy, Set-theoretic solutions of the {Y}ang-{B}axter equation,
  {RC}-calculus, and {G}arside germs, Adv. Math. 282 (2015) 93--127.
\newline\urlprefix\url{https://doi.org/10.1016/j.aim.2015.05.008}

\bibitem{Dr90}
V.~G. Drinfel{\cprime}d, On some unsolved problems in quantum group theory, in:
  Quantum groups ({L}eningrad, 1990), vol. 1510 of Lecture Notes in Math.,
  Springer, Berlin, 1992, pp. 1--8.
\newline\urlprefix\url{http://dx.doi.org/10.1007/BFb0101175}

\bibitem{EtGe98}
P.~Etingof, S.~Gelaki, A method of construction of finite-dimensional
  triangular semisimple {H}opf algebras, Math. Res. Lett. 5~(4) (1998)
  551--561.
\newline\urlprefix\url{https://doi.org/10.4310/MRL.1998.v5.n4.a12}

\bibitem{ESS99}
P.~Etingof, T.~Schedler, A.~Soloviev, Set-theoretical solutions to the quantum
  {Y}ang-{B}axter equation, Duke Math. J. 100~(2) (1999) 169--209.
\newline\urlprefix\url{http://dx.doi.org/10.1215/S0012-7094-99-10007-X}

\bibitem{Ga12}
T.~Gateva-Ivanova, Quadratic algebras, {Y}ang-{B}axter equation, and
  {A}rtin-{S}chelter regularity, Adv. Math. 230~(4-6) (2012) 2152--2175.
\newline\urlprefix\url{https://doi.org/10.1016/j.aim.2012.04.016}

\bibitem{GaB98}
T.~Gateva-Ivanova, M.~Van~den Bergh, Semigroups of {$I$}-type, J. Algebra
  206~(1) (1998) 97--112.
\newline\urlprefix\url{http://dx.doi.org/10.1006/jabr.1997.7399}

\bibitem{GVe17}
L.~Guarnieri, L.~Vendramin, Skew braces and the {Y}ang-{B}axter equation, Math.
  Comp. 86~(307) (2017) 2519--2534.
\newline\urlprefix\url{https://doi.org/10.1090/mcom/3161}

\bibitem{Hi83}
J.~B. Hickey, Semigroups under a sandwich operation, Proc. Edinburgh Math. Soc.
  (2) 26~(3) (1983) 371--382.
\newline\urlprefix\url{https://doi.org/10.1017/S0013091500004442}

\bibitem{Ho76book}
J.~M. Howie, An introduction to semigroup theory, Academic Press [Harcourt
  Brace Jovanovich, Publishers], London-New York, 1976, l.M.S. Monographs, No.
  7.

\bibitem{Ho95book}
J.~M. Howie, Fundamentals of semigroup theory, vol.~12 of London Mathematical
  Society Monographs. New Series, The Clarendon Press, Oxford University Press,
  New York, 1995, oxford Science Publications.

\bibitem{JVA19}
E.~Jespers, A.~Van~Antwerpen, Left semi-braces and solutions of the
  {Y}ang-{B}axter equation, Forum Math. 31~(1) (2019) 241--263.
\newline\urlprefix\url{https://doi.org/10.1515/forum-2018-0059}

\bibitem{Ka95}
C.~Kassel, Quantum groups, vol. 155 of Graduate Texts in Mathematics,
  Springer-Verlag, New York, 1995.
\newline\urlprefix\url{https://doi.org/10.1007/978-1-4612-0783-2}

\bibitem{kauffman2004braiding}
L.~H. Kauffman, S.~J. Lomonaco~Jr, Braiding operators are universal quantum
  gates, New Journal of Physics 6~(1) (2004) 134.

\bibitem{KoTr20}
A.~Koch, P.~J. Truman, Opposite skew left braces and applications, J. Algebra
  546 (2020) 218--235.
\newline\urlprefix\url{https://doi.org/10.1016/j.jalgebra.2019.10.033}

\bibitem{Ku83}
M.~Kunze, Zappa products, Acta Math. Hungar. 41~(3-4) (1983) 225--239.
\newline\urlprefix\url{https://doi.org/10.1007/BF01961311}

\bibitem{Le17}
V.~Lebed, Cohomology of idempotent braidings with applications to factorizable
  monoids, Internat. J. Algebra Comput. 27~(4) (2017) 421--454.
\newline\urlprefix\url{https://doi.org/10.1142/S0218196717500229}

\bibitem{LuYZ00}
J.-H. Lu, M.~Yan, Y.-C. Zhu, On the set-theoretical {Y}ang-{B}axter equation,
  Duke Math. J. 104~(1) (2000) 1--18.
\newline\urlprefix\url{http://dx.doi.org/10.1215/S0012-7094-00-10411-5}

\bibitem{MaShi18}
D.~K. Matsumoto, K.~Shimizu, Quiver-theoretical approach to dynamical
  {Y}ang-{B}axter maps, J. Algebra 507 (2018) 47--80.
\newline\urlprefix\url{https://doi.org/10.1016/j.jalgebra.2018.04.003}

\bibitem{Mi18}
M.~M. Miccoli, Almost semi-braces and the {Y}ang--{B}axter equation, Note Mat.
  38~(1) (2018) 83--88.

\bibitem{P06Y}
J.~H. Perk, H.~Au-Yang, {Y}ang-{B}axter equations, Encyclopedia of Mathematical
  Physics, Vol. 5, (Elsevier Science, Oxford, 2006) (2006) 465--473.

\bibitem{PeRe99}
M.~Petrich, N.~R. Reilly, Completely regular semigroups, vol.~23 of Canadian
  Mathematical Society Series of Monographs and Advanced Texts, John Wiley \&
  Sons, Inc., New York, 1999, a Wiley-Interscience Publication.

\bibitem{Ra12}
D.~E. Radford, Hopf algebras, vol.~49 of Series on Knots and Everything, World
  Scientific Publishing Co. Pte. Ltd., Hackensack, NJ, 2012.

\bibitem{Ru07a}
W.~Rump, Braces, radical rings, and the quantum {Y}ang-{B}axter equation, J.
  Algebra 307~(1) (2007) 153--170.
\newline\urlprefix\url{http://dx.doi.org/10.1016/j.jalgebra.2006.03.040}

\bibitem{Ru19}
W.~Rump, Set-theoretic solutions to the {Y}ang-{B}axter equation, skew-braces,
  and related near-rings, J. Algebra Appl. 18~(8) (2019) 1950145, 22.
\newline\urlprefix\url{https://doi.org/10.1142/S0219498819501457}

\bibitem{Sm18}
A.~Smoktunowicz, On {E}ngel groups, nilpotent groups, rings, braces and the
  {Y}ang-{B}axter equation, Trans. Amer. Math. Soc. 370~(9) (2018) 6535--6564.
\newline\urlprefix\url{https://doi.org/10.1090/tran/7179}

\bibitem{So00}
A.~Soloviev, Non-unitary set-theoretical solutions to the quantum
  {Y}ang-{B}axter equation, Math. Res. Lett. 7~(5-6) (2000) 577--596.
\newline\urlprefix\url{http://dx.doi.org/10.4310/MRL.2000.v7.n5.a4}

\bibitem{StVo20x}
D.~Stanovsk{\`y}, P.~Vojt{\v{e}}chovsk{\`y}, Idempotent solutions of the
  {Y}ang-{B}axter equation and twisted group division, arXiv:2002.02854.

\bibitem{Ya67}
C.~N. Yang, Some exact results for the many-body problem in one dimension with
  repulsive delta-function interaction, Phys. Rev. Lett. 19 (1967) 1312--1315.
\newline\urlprefix\url{https://doi.org/10.1103/PhysRevLett.19.1312}

\bibitem{zhang2005universal}
Y.~Zhang, L.~H. Kauffman, M.-L. Ge, Universal quantum gate, yang--baxterization
  and hamiltonian, International Journal of Quantum Information 3~(04) (2005)
  669--678.

\end{thebibliography}
%%% Remove comment to use the external .bib file (using bibtex).


\end{document}